\documentclass[10pt]{article}
\usepackage[english]{babel}
\usepackage{amsfonts}
\usepackage{a4wide}
\usepackage{color}
\newtheorem{lemma}{Lemma}[section]
\newtheorem{theorem}{Theorem}[section]
\newtheorem{proposition}{Proposition}[section]

\newtheorem{corollary}{Corollary}[section]

\newenvironment{proof}{{\noindent\bf Proof\ }}{$\diamondsuit$ \\}
\newenvironment{remark}{{\noindent\bf Remark\ }}

\begin{document}
\title{Factorization identities for reflected processes, with applications}
\author{Brian H. Fralix, Johan S.H. van Leeuwaarden and Onno J. Boxma \footnote{BF (corresponding author) is with Clemson University, Department of Mathematical Sciences, O-110 Martin Hall, Box 340975, Clemson, SC 29634, USA. Email: {\tt bfralix@clemson.edu}.
JvL and OB are with Eindhoven University of Technology, Department of Mathematics and Computer Science, Eindhoven, The Netherlands. Emails: {\tt j.s.h.v.leeuwaarden@tue.nl}, {\tt o.j.boxma@tue.nl}}}
\maketitle

\begin{abstract} We derive factorization identities for a class of preemptive-resume queueing systems, with batch arrivals and catastrophes that, whenever they occur, eliminate multiple customers present in the system.  These processes are quite general, as they can be used to approximate L\'evy processes, diffusion processes, and certain types of growth-collapse processes;  thus, all of the processes mentioned above also satisfy similar factorization identities.  In the L\'evy case, our identities simplify to both the well-known Wiener-Hopf factorization, and another interesting factorization of reflected L\'evy processes starting at an arbitrary initial state.  We also show how the ideas can be used to derive transforms for some well-known state-dependent/inhomogeneous birth-death processes and diffusion processes.
\end{abstract}



\section{Introduction}

The Wiener-Hopf factorization is a classical result in both the theory of random walks and the theory of L\'evy processes.  For a L\'evy process $X$, the factorization allows us to write the position of $X$ at an independent exponential time $e_q$, i.e. $X(e_q)$, as the sum of two independent random variables:  $\inf_{0 \leq s \leq e_q}X(s)$ and $X(e_q) - \inf_{0 \leq s \leq e_q}X(s)$, with the latter random variable representing the reflection of $X$ at a random time $e_q$.  In principle, the distribution of the reflected process at time $e_q$ can be derived if and only if the distribution of the infimum of $X$ over $[0,e_q]$ is known as well.

We show that a similar type of property is also found in processes that may not necessarily be expressible as a reflection of a simpler process. To do this, we introduce the Preemptive-Resume Production system, or PRP system, and we show that it satisfies a factorization identity.  Technically, for an arbitrary PRP system the identity is not a true factorization, but it is in some cases:  when $X$ is a L\'evy process, for instance, our factorization identity is equivalent to the Wiener-Hopf factorization.  The notion of a PRP system may appear at first to be somewhat contrived, but this is not the case:  such systems can be used to approximate many types of important processes found in the probability literature, such as L\'evy processes, diffusion processes, and even Markovian growth-collapse models.

Our factorization results also provide insight into the time-dependent behavior of a number of important birth-death processes, with birth/death rates that may depend on the state of the system.  For instance, our Wiener-Hopf identity shows how the probability mass function of the $M/M/s$ queue-length at an independent exponential time $e_q$ can be expressed entirely in terms of quantities from a $M/M/1$ queue and a $M/M/\infty$ queue.  Similarly, a $M/M/s/K$ queue (assuming $s < K$, otherwise trivial) can be expressed in terms of a $M/M/\infty$ queue and a $M/M/1/(K-s)$ queue, and a similar observation may be made for Markovian queues with reneging.  In particular, the pmf for the $M/M/s/K$ queue can be quickly derived from the solutions to the $M/M/\infty$ queue and the $M/M/1/(K-s)$ queue, without having to make use of the Kolmogorov forward equations corresponding to the $M/M/s/K$ queue.  Similar expressions can also be derived for diffusions that can be expressed as limits of birth-death processes.

Readers wondering why we are interested in studying the distribution of $X(e_q)$ should note that $P(X(e_{q}) = k)$ can be expressed as $q$ times the Laplace transform of the function $P(X(t) = k)$ evaluated at $q$, where $q$ is a positive real number.  Hence, having knowledge of $X(e_q)$ yields insight into the behavior of $X(t)$, for each $t \geq 0$.  Even though we restrict ourselves to the case where $q$ is real and positive, it is possible to derive similar transform expressions for the function $P(X(t) = k)$ at complex numbers with positive real part:  readers will find explanations of how to make such extensions at various places throughout the paper, whenever they are needed.

The factorization results we present here seem to be somewhat related to those found in Millar \cite{Millar}.  The main result of \cite{Millar} establishes that for a Markov process $X$ satisfying suitable regularity conditions, the distribution of the path of $X$ from the time at which a functional of it attains a minimum is independent of the behavior of $X$ before having attained this minimum.  Contrary to \cite{Millar}, our factorization results are valid for processes that are not necessarily Markovian, and our results also show how various transforms associated with some processes can be decomposed into computable transforms associated with other types of simpler stochastic processes, as previously mentioned.

\section{Model Description}

We now define what we refer to as a Preemptive-Resume Production system, or PRP system.  At time zero there are a countably infinite number of customers present, which are labeled $n_0, n_0 -1, n_0 -2, n_0 -3, \ldots$.  The system then begins to process the work of the customer that possesses the highest label, or number, which at time zero is customer $n_0$.  The server processes jobs in accordance to the Last-Come-First-Served Preemptive-Resume discipline.  All customers possess a random, generally distributed amount of work, and the amount of work possessed by a given customer is independent of the amounts of work of all other customers that will visit, or have visited the system.  We are interested in studying the process $Q := \{Q(t); t \geq 0 \}$, where $Q(t)$ represents the label of the customer being served by the server at time $t$:  for example, $Q(0) = n_0$.

There are two sets of Poisson processes governing arrivals to the production system.  The first set governs single arrivals to the system, and consists of an independent collection of Poisson processes $\{A_{0, j}\}_{j \in \mathbb{Z}}$, where $A_{0,j}$ has rate $\lambda_{0,j}$.  At an arbitrary time $t$, when $Q(t-) = j$, we say that $A_{0,j}$ is active:  in other words, if a point of $A_{0,j}$ occurs at time $t$ while $Q(t-) = j$, then $Q(t) = j + 1$, and the new arrival is immediately given label $j + 1$.  Otherwise, the point of $A_{0,j}$ occurring at time $t$ is ignored if $Q(t-) = k \neq j$, so no new customer arrives to the system at that time.  Once the server finishes with the customer having label $j+1$, it begins serving customer $j$, returning to where it left off before previously departing.

The second set of Poisson processes govern batch arrivals of customers to the system (we allow batches to be of size one).  This second set consists of an independent collection of Poisson processes $\{A_{1,j,k}\}_{j,k \in \mathbb{Z}}$, where $A_{1,j,k}$ has rate $\lambda_{1,j}P(Z_{1,j} = k - j)$.  Again, while $Q(t-) = j$, we say that the subcollection $\{A_{1,j,k}\}_{k \in \mathbb{Z}}$ is active, so a point of $A_{1,j,k}$ at time $t$ pushes $Q$ from level $j$ to level $k$, the $k - j$ customers in the batch are instantaneously assigned labels $j+1$, $j + 2$, \ldots, $k$, and the server immediately begins processing customer $k$.  Here $Z_{1,j}$ is a generic random variable representing the jump size of the $Q$ process from level $j$:  we allow the distribution of these jumps to depend on the current level.

We further assume that catastrophes occur according to a modulated Poisson process $D := \{D(t); t \geq 0\}$, with rate $\delta_{Q(t-)}$.  At the time of a catastrophe, a random number of customers are removed from the system:  in particular, if $Q(t-) = n$, and a catastrophe occurs at time $t$, which eliminates $k$ customers, then customers $n, n-1, n-2, \ldots, n-k+1$ are immediately removed from the system, and at time $t$ the server begins to process the remaining amount of work possessed by customer $n-k$, and so $Q(t) = n-k$.  We assume that the distribution function of the number of removals at time $t$ depends on $Q(t-)$, so that the downward jump distribution of the process may depend on the level of the process, immediately before a jump.

Readers may wonder why we chose to use an infinite collection of independent Poisson processes to govern arrivals to our queueing system, while not modeling catastrophes in the same manner.  The answer lies in the proof of our main result, as modeling the arrival processes in this way allows us to derive a linear system of equations in a most efficient manner.  Indeed, catastrophes can be modeled in the same way, but these will not play as important a role in our proofs.  Our use of collections of Poisson processes to model the arrival process was inspired by Chapter 9 of Br\'emaud \cite{BremaudMC}, who makes use of such a framework when constructing continuous-time Markov chains.  Readers wishing to rigorously construct our PRP systems in the same manner can follow the procedure given there, by expanding the state space of the PRP system to include the residual service time of each customer in the system, thus making it a stochastic recursive system, and Markovian:  readers should note that customers in the system possess generally distributed amounts of work, meaning $\{Q(t); t \geq 0\}$ is not a Markov process unless the state space is expanded to include the residual service times.

Later we will use these processes to approximate L\'evy processes:  arrivals from the $\{A_{0,j}\}_{j}$ collection and service completions of the server will be used to construct Brownian motion, while the batch arrivals and catastrophe processes will be used to construct Compound Poisson processes.

Finally, we also consider a `reflected' PRP system $\{Q_{l}(t); t \geq 0\}$, where $l$ is a fixed integer.  This system behaves in a similar manner as $Q$, with the following exception:  whenever $Q_{l}$ is in a state $i$, and a catastrophe occurs which, in the original system, would place $Q$ at a level at or lower than $l$, $Q_{l}$ instead makes a transition from state $i$ to state $l$.  When $Q_l$ is at level $l$, the server stops working until the next arrival:  hence, customer $l$ is in the system for all time.  Finally, upward jumps of $Q_{l}$ behave the same as upward jumps of $Q$.  We refer to $Q_{l}$ as a reflected PRP system with reflection at level $l$.

\section{Main Results}

Our main result establishes that the process $\{Q(t); t \geq 0\}$ from the PRP system satisfies a factorization identity, which we now give.
\begin{theorem} \label{mainresult} Let $e_q$ be an exponential random variable with rate $q > 0$, independent of $Q$.  For any two integers $k, l$, where $k \geq 0$ and $l \leq n_0 = Q(0)$,
\begin{eqnarray*} P(Q_{l}(e_q) = k + l \mid Q_{l}(0) = l) &=& P(Q(e_q) = k + l \mid \inf_{0 \leq u \leq e_q}Q(u) = l) \\ &=& P(Q(e_q) - \inf_{0 \leq u \leq e_q}Q(u) = k \mid \inf_{0 \leq u \leq e_q}Q(u) = l). \end{eqnarray*} \end{theorem}

\begin{proof} To help readers understand the proof, we break it up into three steps.

\noindent \textbf{Step 1}  We begin by presenting the following identity, which is satisfied by the sample paths of our PRP system:  for each $t \geq 0$, we see that for any two integers $k,l$ with $k \geq 1$, $l \leq n_0 = Q(0)$,
\begin{eqnarray} \label{mainindicator} & & \textbf{1}(Q(t) \geq k + l, \inf_{0 \leq u \leq t}Q(u) = l) \nonumber \\ &=& \int_{0}^{t}\textbf{1}(Q(s-) = k - 1 + l, \inf_{0 \leq u < s}Q(u) = l)\textbf{1}(\inf_{u \in [s,t]}Q(s) \geq k + l)A_{0,k-1+l}(ds) \nonumber \\ &+& \sum_{j=0}^{k-1}\sum_{m = k}^{\infty}\int_{0}^{t}\textbf{1}(Q(s-) = j + l, \inf_{0 \leq u < s}Q(u) = l)\textbf{1}(\inf_{u \in [s,t]} Q(u) \geq k + l)A_{1,j + l, m + l}(ds). \end{eqnarray}
The identity (\ref{mainindicator}) says that, in order that $Q(t) \geq k + l$, exactly one of two things must happen: if the infimum of the process over $[0,t]$ is $l$, either (i) there exists a time point $s \leq t$ such that $Q(s-) = k-1 + l$, $Q(s) = k + l$ (due to the arrival of a customer from $A_0$ at time $s$), and the process stays at or above level $k + l$ in $[s,t]$, giving the first term, or (ii) there exists a time point $s \leq t$ such that, due to a batch of customers arriving at time $s$ (which is contributed by $A_1$), the process crosses level $k + l$, reaching some level at or above $k + l$ at time $s$, and stays at or above $k + l$ during $[s,t]$, giving the second term.

After taking expected values of both sides of (\ref{mainindicator}), we get
\begin{eqnarray} \label{firstcalculation1} & & P(Q(t) \geq k + l, \inf_{0 \leq u \leq t}Q(u) = l) \nonumber \\ &=& E\left[\int_{0}^{t}\textbf{1}(Q(s-) = k-1 + l, \inf_{0 \leq u < s}Q(u) = l)\textbf{1}(\inf_{u \in [s,t]}Q(u) \geq k + l)A_{0,k-1 + l}(ds)\right] \nonumber \\ &+& \sum_{j=0}^{k-1}\sum_{m=k}^{\infty}E\left[\int_{0}^{t}\textbf{1}(Q(s-) = j + l, \inf_{0 \leq u < s}Q(u) = l)\textbf{1}(\inf_{u \in [s,t]}Q(u) \geq k + l)A_{1,j + l, m + l}(ds)\right]. \end{eqnarray}  We can use the Campbell-Mecke formula to evaluate the expected values found on the right-hand side of Equation (\ref{firstcalculation1}).  Notice first that
\begin{eqnarray*} & & E\left[\int_{0}^{t}\textbf{1}(Q(s-) = k-1 + l, \inf_{0 \leq u < s}Q(u) = l)\textbf{1}(\inf_{u \in [s,t]}Q(u) \geq k + l)A_{0,k-1 + l}(ds)\right] \\ &=& \lambda_{0,k-1+l}\int_{0}^{t}\mathcal{P}_{s}(Q(s-) = k-1 + l, \inf_{0 \leq u < s}Q(u) = l, \inf_{u \in [s,t]}Q(u) \geq k + l)ds, \end{eqnarray*} where $\mathcal{P}$ represents the Palm kernel induced by $A_{0,k-1 + l}$.  Furthermore, since the server processes work in a preemptive-resume manner, we can also use the Campbell-Mecke formula to establish that
\begin{eqnarray*} \label{goodpreemptiveresumecalc} & & \mathcal{P}_{s}(\inf_{u \in [s,t]}Q(u) \geq k + l,  Q(s-) = k-1 + l, \inf_{0 \leq u < s}Q(u) = l) \nonumber \\ &=& P(\tau_{k + l,k + l} > t-s)\mathcal{P}_{s}(Q(s-) = k-1 + l, \inf_{0 \leq u < s}Q(u) = l) \end{eqnarray*} where $\tau_{k,j}$ is the amount of time it takes the PRP system to go below state $j$, starting from state $k$, $j \leq k$, where all customers labeled $j,j+1, \ldots, k$ have not yet received any attention from the server.  Moreover, if we let $\{\mathcal{F}_{t}; t \geq 0\}$ represent the minimal filtration induced by $Q$ and our arrival and catastrophe processes, we see that the event $\{Q(s-) = k-1 + l, \inf_{0 \leq u < s}Q(u) = l\} \in \mathcal{F}_{s-}$, and so Proposition \ref{ASTAprop} in the Appendix yields
\begin{eqnarray*} \mathcal{P}_{s}(Q(s-) = k-1 + l, \inf_{0 \leq u < s}Q(u) = l) = P(Q(s) = k-1 + l, \inf_{0 \leq u \leq s}Q(u) = l). \end{eqnarray*}  An analogous argument can be used to evaluate the second type of expectation found in (\ref{firstcalculation1}).  Plugging these expressions into (\ref{firstcalculation1}) gives
\begin{eqnarray} \label{firstcalculation2} & & P(Q(t) \geq k + l, \inf_{0 \leq u \leq t}Q(u) = l) = \lambda_{0, k-1 + l}\int_{0}^{t}P(Q(s) = k-1 + l, \inf_{0 \leq u \leq s}Q(u) = l)P(\tau_{k + l,k + l} > t-s)ds \nonumber \\ &+& \sum_{j=0}^{k-1}\sum_{m=k}^{\infty}\lambda_{1,j + l}P(Z_{1,j + l} = m-j)\int_{0}^{t}P(\tau_{m + l, k + l} > t-s)P(Q(s-) = j + l, \inf_{0 \leq u \leq s}Q(u) = l)ds. \end{eqnarray}  After integrating both sides of (\ref{firstcalculation2}) with respect to an exponential density with rate $q > 0$, we get
\begin{eqnarray*} & & P(Q(e_q) \geq k + l, \inf_{0 \leq u \leq e_q} Q(u) = l) = \lambda_{0,k-1 + l}\frac{(1 - \phi_{k + l,k + l}(q))}{q}P(Q(e_q) = k-1 + l, \inf_{0 \leq u \leq e_q}Q(u) = l) \\  &+& \sum_{j=0}^{k-1}\sum_{m=k}^{\infty}\lambda_{1,j + l} P(Z_{1,j + l} = m-j) \frac{(1 - \phi_{m + l, k + l}(q))}{q}P(Q(e_q) = j + l, \inf_{0 \leq u \leq e_q}Q(u) = l) \end{eqnarray*} where $\phi_{m + l,k + l}$ represents the Laplace-Stieltjes transform of $\tau_{m + l,k + l}(0)$ (with $Q(0) = m + l$).  Dividing by $P(\inf_{0 \leq u \leq e_q}Q(u) = l)$ finally yields
\begin{eqnarray} \label{infequations} & & P(Q(e_q) \geq k + l \mid \inf_{0 \leq u \leq e_q} Q(u) = l) = \lambda_{0,k-1 + l}\frac{(1 - \phi_{k + l,k + l}(q))}{q}P(Q(e_q) = k-1 + l \mid \inf_{0 \leq u \leq e_q}Q(u) = l) \nonumber \\  &+& \sum_{j=0}^{k-1}\sum_{m=k}^{\infty}\lambda_{1,j + l} P(Z_{1,j+l} = m - j) \frac{(1 - \phi_{m + l, k + l}(q))}{q}P(Q(e_q) = j + l \mid \inf_{0 \leq u \leq e_q}Q(u) = l). \end{eqnarray}

\noindent \textbf{Step 2} We now show that the system of equations (\ref{infequations}) has a unique solution.  Notice that for a fixed integer $l$, these equations can be iteratively solved, since
\begin{eqnarray*} \sum_{k=0}^{\infty}P(Q(e_q) = k + l \mid \inf_{0 \leq u \leq e_q}Q(u) = l) = 1.  \end{eqnarray*}  Indeed, notice that
\begin{eqnarray*} & & 1 - P(Q(e_q) = l \mid \inf_{0 \leq u \leq e_q}Q(u) = l) = P(Q(e_q) \geq l + 1 \mid \inf_{0 \leq u \leq e_q}Q(s) = l) \\ &=& \lambda_{0,l}\frac{(1 - \phi_{l+1,l+1}(q))}{q}P(Q(e_q) = l \mid \inf_{0 \leq u \leq e_q}Q(u) = l) \nonumber \\  &+& \sum_{m=1}^{\infty}\lambda_{1,l} P(Z_{1,l} = m) \frac{(1 - \phi_{m + l, 1 + l}(q))}{q}P(Q(e_q) = l \mid \inf_{0 \leq u \leq e_q}Q(u) = l) \end{eqnarray*} which allows us to determine $P(Q(e_q) = l \mid \inf_{0 \leq u \leq e_q}Q(u) = l)$, and all other probabilities can be determined in a similar, iterative manner.  Hence, there is a unique probability measure on the integers that satisfies these equations.

\noindent \textbf{Step 3} By precisely the same arguments, we see that the $Q_{l}$ process satisfies the same system of equations.  Indeed, when $Q_{l}(0) = l$,
\begin{eqnarray*} & & P(Q_{l}(e_q) \geq k + l) = \lambda_{0, k + l - 1}\frac{1 - \phi_{k+l,k+l}(q)}{q}P(Q_{l}(e_q) = k - 1 + l) \\ &+& \sum_{j=0}^{k-1}\sum_{m=k}^{\infty}\lambda_{1,l + j}P(Z_{1,l + j} = m - j)\frac{1 - \phi_{m + l, k + l}(q)}{q}P(Q_{l}(e_q) = j + l). \end{eqnarray*}  Thus, we see that
\begin{eqnarray*} P(Q_{l}(e_q) = k + l \mid Q_{l}(0) = l) = P(Q(e_q) = k + l \mid \inf_{0  \leq s \leq e_q}Q(s) = l) \end{eqnarray*} completing the proof. \end{proof}

\begin{remark} It is worth noting, from the point of view of numerical transform inversion \cite{AbateWhittInversion}, that a similar result can be derived when we consider complex-valued $q$, i.e. expressions of the form
\begin{eqnarray*} \int_{0}^{\infty}P(Q(t) = k + l, \inf_{0 \leq s \leq t}Q(s) = l)qe^{-qt}dt \end{eqnarray*} for complex $q$ with positive real part, i.e. those $q$ satisfying $\Re(q) > 0$, as opposed to $P(Q(e_q) = k + l, \inf_{0 \leq s \leq e_q}Q(s) = l)$ for real $q > 0$.
First note that for $q = x + iy$ satisfying $\Re(q) = x > 0$, with $e_{x}$ being exponential with rate $x$, independent of $Q$,
\begin{eqnarray*} \int_{0}^{\infty}P(Q(t) = k + l, \inf_{0 \leq s \leq t}Q(s) = l)qe^{-qt}dt &=& \int_{0}^{\infty}P(Q(t) = k + l, \inf_{0 \leq s \leq t}Q(s) = l)(x + iy)e^{-iyt}e^{-xt}dt \\ &=& \frac{(x + iy)}{x}E[\textbf{1}(Q(e_{x}) = k + l, \inf_{0 \leq s \leq e_{x}}Q(s) = l)e^{-iye_{x}}]. \end{eqnarray*}  Using this observation, we can mimic the proof of Theorem \ref{mainresult} in a straightforward manner to determine that
\begin{eqnarray*} & & E[\textbf{1}(Q(e_{x}) = k+l)e^{-iye_{x}} \mid \inf_{0 \leq s \leq e_{x}}Q(s) = l, Q(0) = n_0] \\ &=& \frac{E[e^{-iye_{x}}\textbf{1}(\inf_{0 \leq s \leq e_{x}}Q(s) = l) \mid Q(0) = n_0]E[\textbf{1}(Q_{l}(e_{x}) = k + l)e^{-iye_{x}} \mid Q_{l}(0) = l]}{P(\inf_{0 \leq s \leq e_{x}}Q(s) = l \mid Q(0) = n_{0})}\frac{x + iy}{x} \end{eqnarray*} which contains quantities that are given in terms of either the reflection $Q_{l}$ reflected at $l$, or hitting-time transforms associated with the original process $Q$.  To see why only these types of transforms need to be computed, note that letting $\tau_{l} = \inf\{t \geq 0:  Q(t) \leq l\}$ yields
\begin{eqnarray*} E[e^{-iye_{x}}\textbf{1}(\inf_{0 \leq u \leq e_{x}}Q(u) = l) \mid Q(0) = n_0] &=& E[e^{-iye_{x}}\textbf{1}(\inf_{0 \leq u \leq e_{x}}Q(u) \leq l) \mid Q(0) = n_{0}] \\ &-& E[e^{-iye_{x}}\textbf{1}(\inf_{0 \leq u \leq e_{x}}Q(u) \leq l-1) \mid Q(0) = n_0] \\ &=& E[e^{-iye_{x}}\textbf{1}(\tau_{l} \leq e_{x}) \mid Q(0) = n_{0}] \\ &-& E[e^{-iye_{x}}\textbf{1}(\tau_{l-1} \leq e_{x}) \mid Q(0) = n_0] \\ &=& \frac{x}{x + iy}E[e^{-iy\tau_{l}}\textbf{1}(\tau_{l} \leq e_{x}) \mid Q(0) = n_0] \\ &-& \frac{x}{x + iy}E[e^{-iy \tau_{l-1}}\textbf{1}(\tau_{l-1} \leq e_{x}) \mid Q(0) = n_0] \\ &=& \frac{x}{x + iy}\left[E[e^{-q\tau_{l}} \mid Q(0) = n_0] - E[e^{-q \tau_{l-1}} \mid Q(0) = n_0]\right]  \end{eqnarray*}  This gives
\begin{eqnarray*} & & E[\textbf{1}(Q(e_{x}) = k + l)e^{-iye_{x}} \mid \inf_{0 \leq s \leq e_{x}}Q(s) = l, Q(0) = n_0] \\ &=& \frac{\left[E[e^{-q\tau_{l}} \mid Q(0) = n_0] - E[e^{-q \tau_{l-1}} \mid Q(0) = n_0]\right]}{\left[E[e^{-x\tau_{l}} \mid Q(0) = n_0] - E[e^{-x \tau_{l-1}} \mid Q(0) = n_0]\right]}E[\textbf{1}(Q_{l}(e_{x}) = k + l)e^{-iye_{x}} \mid Q_{l}(0) = l] \end{eqnarray*} implying
\begin{eqnarray*} & & \int_{0}^{\infty}P(Q(t) = k + l, \inf_{0 \leq s \leq t}Q(s) = l \mid Q(0) = n_0)qe^{-qt}dt \\ &=& \left[E[e^{-q\tau_{l}} \mid Q(0) = n_0] - E[e^{-q \tau_{l-1}} \mid Q(0) = n_0]\right] \int_{0}^{\infty}P(Q_{l}(t) = k+l \mid Q_{l}(0) = l)qe^{-qt}dt \end{eqnarray*} which is clearly the complex analogue of the formula given in Theorem \ref{mainresult}.  All other types of transforms that we will need can be computed in a similar manner, for complex $q$.
\end{remark}

We now show that the reflected process $\{Q_{0}(t); t \geq 0\}$ exhibits a similar type of factorization identity.
\begin{theorem} \label{mainreflectedresult} Suppose $Q$ is a PRP system with $Q(0) = n_0$, and let $Q_{0}$ be the reflected version of $Q$ at level zero, with $Q_{0}(0) = n_0$.   Then for each integer $l \geq 0$, and each integer $k \geq 1$,
\begin{eqnarray*} P(Q(e_q) - \inf_{0 \leq u \leq e_q}Q(u) = k \mid \inf_{0 \leq u \leq e_q}Q(u) = l) &=& P(Q_{0}(e_q) - \inf_{0 \leq u \leq e_q}Q_{0}(u) = k \mid \inf_{0 \leq u \leq e_q}Q_{0}(u) = l). \end{eqnarray*} \end{theorem}
\begin{proof} Notice that a sample-path identity that is completely analogous to (\ref{mainindicator}) can be established for $Q_{0}$:  for each $l \geq 0$, $k \geq 1$,
\begin{eqnarray} \label{secondmainindicator} & & \textbf{1}(Q_{0}(t) \geq k + l, \inf_{0 \leq u \leq t}Q_{0}(u) = l) \nonumber \\ &=& \int_{0}^{t}\textbf{1}(Q_{0}(s-) = k-1 + l, \inf_{0 \leq u \leq s}Q_{0}(u) = l)\textbf{1}(\inf_{u \in [s,t]}Q_{0}(u) = k + l)A_{0,k-1 + l}(ds) \nonumber \\ &+& \sum_{j=0}^{k-1}\sum_{m = k}^{\infty}\int_{0}^{t}\textbf{1}(Q_{0}(s-) = j + l, \inf_{0 \leq u < s}Q_{0}(u) = l)\textbf{1}(\inf_{u \in [s,t]}Q_{0}(u) \geq k + l)A_{1,j + l, m + l}(ds). \end{eqnarray}
\noindent Applying the same steps found in Step 1 of the proof of Theorem \ref{mainresult} yields
\begin{eqnarray} \label{secondinfequations} & & P(Q_{0}(e_q) \geq k + l \mid \inf_{0 \leq u \leq e_q} Q_{0}(u) = l) = \lambda_{0,k-1 + l}\frac{(1 - \phi_{k + l,k + l}(q))}{q}P(Q_{0}(e_q) = k-1 + l \mid \inf_{0 \leq u \leq e_q}Q_{0}(u) = l) \nonumber \\  &+& \sum_{j=0}^{k-1}\sum_{m=k}^{\infty}\lambda_{1,j + l} P(Z_{1,j + l} = m - j) \frac{(1 - \phi_{m + l, k + l}(q))}{q}P(Q_{0}(e_q) = j + l \mid \inf_{0 \leq u \leq e_q}Q_{0}(u) = l). \end{eqnarray}  For our fixed $l$, we notice that the equations that form system (\ref{infequations}) are the same as the equations found in (\ref{secondinfequations}).  Hence, by the uniqueness result proven in Step 2 of Theorem \ref{mainresult} we have
\begin{eqnarray*} P(Q(e_q) \geq k + l \mid \inf_{0 \leq u \leq e_q}Q(u) = l) = P(Q_{0}(e_q) \geq k + l \mid \inf_{0 \leq u \leq e_q} Q_{0}(u) = l) \end{eqnarray*}  which completes the proof. \end{proof}

Two interesting factorization results can be derived, when the batch and catastrophe sizes of both $Q$ and $Q_0$ have distributions that are state-independent.
Clearly, in this case we see that for each $k \geq 0$ and $l$, $P(Q_{l}(e_q) = k + l \mid Q_{l}(0) = l) = P(Q_{0}(e_q) = k \mid Q_{0}(0) = 0)$, and since $Q_{0}$ is the reflection of $Q$ at level 0, we also find that
\begin{eqnarray*}Q_{0}(e_q) \stackrel{d}{=} Q(e_q) - \inf_{0 \leq u \leq e_q}Q(u) \end{eqnarray*} which follows since customers are processed in a Last-Come-First-Served Preemptive-Resume manner.  Hence, Theorem \ref{mainresult} yields for each $k \geq 0$, $l \leq 0 = Q(0)$,
\begin{eqnarray*} P(Q(e_q) - \inf_{0 \leq u \leq e_q}Q(u) = k) = P(Q(e_q) - \inf_{0 \leq u \leq e_q}Q(u) = k \mid \inf_{0 \leq u \leq e_q}Q(u) = l). \end{eqnarray*}  In other words, the following corollary holds.

\begin{corollary} \label{PRPcorollary}  Suppose that $\{Q(t); t \geq 0\}$ represents a PRP system, with state-independent jumps, and let $e_q$ be an exponential random variable with rate $q > 0$, independent of $Q$.  Then for each $\omega \in \mathbb{R}$,
\begin{eqnarray*} E_{0}[e^{i \omega Q(e_q)}] = E_{0}[e^{i \omega \inf_{0 \leq u \leq e_q}Q(u)}]E_{0}[e^{i \omega (Q(e_q) - \inf_{0 \leq u \leq e_q}Q(u))}]. \end{eqnarray*} \end{corollary}  Here $E_{x}$ is the expectation corresponding to $P_{x}$, where $P_{x}$ is a probability measure under the condition that our process starts at level $x$.  This notation will be used in many places throughout the rest of the paper.

This factorization has been well-known for L\'evy processes since the late 60's, due to Percheskii and Rogozin \cite{PercheskiiRogozin}, and the first probabilistic proof of this result was given in Greenwood and Pitman \cite{GreenwoodPitman}.

We can also conclude from Theorem \ref{mainreflectedresult} that for $l \geq 0$, when $Q_{0}(0) = Q(0) = n_0$,
\begin{eqnarray*} P(Q_{0}(e_q) - \inf_{0 \leq u \leq e_q}Q_{0}(u) = k \mid \inf_{0 \leq u \leq e_q}Q_{0}(u) = l) &=& P(Q(e_q) - \inf_{0 \leq u \leq e_q}Q(u) = k) \\ &=& P(Q_{0}(e_q) - \inf_{0 \leq u \leq e_q}Q_{0}(u) = k) \end{eqnarray*} where the second equality follows from the simple fact that the reflection of $Q_0$ at its infimum is equal in distribution to the reflection of $Q$ at its infimum.  Hence, we see that $Q_{0}(e_q) - \inf_{0 \leq u \leq e_q}Q_{0}(u)$ is actually independent of  $\inf_{0 \leq u \leq e_q}Q_{0}(u)$, which gives us another interesting corollary.

\begin{corollary} \label{reflectedPRPcorollary} Suppose that $\{Q_{0}(t); t \geq 0\}$ is a reflected version of our PRP system, reflected at 0.  Then for each $\omega \in \mathbb{R}$, and each integer $n_0 \geq 0$,
\begin{eqnarray*} E_{n_0}[e^{i \omega Q_{0}(e_q)}] = E_{n_0}[e^{i \omega \inf_{0 \leq u \leq e_q}Q_{0}(u)}]E_{0}[e^{i \omega Q_{0}(e_q)}]. \end{eqnarray*} \end{corollary}

Such a factorization result is useful when studying reflected processes starting in an arbitrary initial state.  Corollary \ref{PRPcorollary} shows that, since $\inf_{0 \leq u \leq e_q}Q(u)$ is independent of $Q(e_q) - \inf_{0  \leq u \leq e_q}Q(u)$, the transforms of $Q(e_q)$ and $\inf_{0 \leq u \leq e_q}Q(u)$ can be used to derive the transform of $Q(e_q) - \inf_{0 \leq u \leq e_q}Q(u)$, which represents the distribution of the reflected process, starting in level zero.  Theorem \ref{mainreflectedresult} can then be used to find the distribution of the reflected process, starting in any initial state, since it is clearly equal in distribution to a convolution of the reflected PRP system $Q_0$ starting in level zero, and a truncated version of $\inf_{0 \leq u \leq e_q}Q(u)$.

We are now ready to see how the Wiener-Hopf factorization for L\'evy processes follows as a consequence of our factorization identities for PRP systems, whose arrival rates, service rates, and jump distributions do not depend on the level of the process.

\subsection{The Wiener-Hopf factorization}
We begin with establishing the well-known version of the Wiener-Hopf factorization, for L\'evy processes.
\begin{theorem} \label{classicalLevyThm} Suppose $X$ is a L\'evy process, and let $e_q$ be an exponential random variable, independent of $X$, with rate $q > 0$.  Then $\inf_{0 \leq s \leq e_q}X(s)$ and $X(e_q) - \inf_{0 \leq s \leq e_q}X(s)$ are independent. \end{theorem}
\begin{proof} Suppose first that $\tilde{X}$ is a L\'evy process that consists of only a Brownian component and a compound Poisson component.  In this case, there exists a sequence of PRP systems $\{\tilde{X}_n\}_{n \geq 1}$, such that $\tilde{X}_n$ converges uniformly on compact sets to $\tilde{X}$:  in fact, each $\tilde{X}_{n}$ process is also a L\'evy process.  We omit the details on constructing the $\{\tilde{X}_n\}_{n}$ sequence, as they are somewhat standard:  interested readers can also find them in a previous online version \cite{FralixvanLeeuwaardenBoxma} of the paper.

From Corollary \ref{PRPcorollary}, we see that the Wiener-Hopf factorization is valid for each PRP system with state-independent jumps.  Applying the L\'evy continuity theorem yields, for each $(\omega_{1}, \omega_{2}) \in \mathbb{R}^{2}$,
\begin{eqnarray*} E[e^{i(\omega_{1} \inf_{0 \leq s \leq e_{q}}X(s) + \omega_{2}(X(e_q) - \inf_{0 \leq s \leq e_q}X(s)))}] &=&  \lim_{n \rightarrow \infty}E[e^{i(\omega_{1} \inf_{0 \leq s \leq e_{q}}\tilde{X}_{n}(s) + \omega_{2}(\tilde{X}_{n}(e_q) - \inf_{0 \leq s \leq e_q}\tilde{X}_{n}(s)))}] \\ &=& \lim_{n \rightarrow \infty}  E[e^{i\omega_{1} \inf_{0 \leq s \leq e_{q}}\tilde{X}_{n}(s)}]E[e^{i(\omega_{2}(\tilde{X}_{n}(e_q) - \inf_{0 \leq s \leq e_q}\tilde{X}_{n}(s)))}] \\ &=& E[e^{i \omega_{1} \inf_{0 \leq s \leq e_q}\tilde{X}(s)}]E[e^{i \omega_{2} (\tilde{X}(e_q) - \inf_{0 \leq s \leq e_{q}}\tilde{X}(s))}] \end{eqnarray*}  proving independence.  To derive this result for an arbitrary L\'evy process, use this result in conjunction with the proof of the L\'evy-It\^o decomposition:  again, finer details of this procedure can be found in \cite{FralixvanLeeuwaardenBoxma}. \end{proof}

Our idea of proving a factorization result for a special type of process, then taking limits is similar to the older approaches of proving the Wiener-Hopf factorization, along with related results:  see for instance Percheskii and Rogozin \cite{PercheskiiRogozin}, along with Gusak and Korolyuk \cite{GusakKorolyuk}.  Our approach differs in the fact that we use a discrete state space in continuous time:  this allows us to state a simple sample-path identity, from which we derive a linear system of equations that has a unique solution.  Moreover, our limiting argument makes use of classical heavy-traffic results from queueing theory.  Readers interested in learning more about classical approaches towards proving the Wiener-Hopf factorization are referred to the recent paper of Kuznetsov \cite{Kuznetsov}.

\subsection{An analogous factorization for the reflection}

We now show how to use Corollary \ref{reflectedPRPcorollary} to deduce an analogous factorization for reflected L\'evy processes, with an arbitrary initial state.

\begin{theorem} \label{reflectionLevyThm} Suppose $X$ represents a L\'evy process, and let $e_q$ be an exponential random variable with rate $q > 0$, independent of $X$.  Moreover, let $R := \{R(t); t \geq 0\}$ represent the reflection of $X$, with a reflected barrier at state zero.  Then, assuming $X(0) = x \geq 0$,
\begin{eqnarray} \label{reflectedlevystatement} E_{x}[e^{i \omega R(e_q)}] &=& E_{0}[e^{i \omega R(e_q)}]E_{x}[e^{i \omega \inf_{0 \leq u \leq e_q}R(u)}]. \end{eqnarray} \end{theorem}
\begin{proof} The proof of this result is completely analogous to the proof of Theorem \ref{classicalLevyThm}.  First, we use Corollary \ref{reflectedPRPcorollary} to establish that it holds for a L\'evy process $X$ that consists of only a Brownian and compound Poisson part.  The general statement then again follows as before, from the proof of the L\'evy-It\^o decomposition.  \end{proof}

Theorem \ref{reflectionLevyThm} can also be derived directly from the Wiener-Hopf factorization.  Here $X(0) = x$, and for each $t \geq 0$
\begin{eqnarray*} R(t) = X(t) - \inf_{0 \leq s \leq t}\min(X(s), 0) \end{eqnarray*} and so
\begin{eqnarray*} R(t) - \inf_{0 \leq s \leq t}R(s) &=& X(t) - \inf_{0 \leq s \leq t}\min(X(s), 0) - \inf_{0 \leq s \leq t}\left((X(s) - \inf_{0 \leq u \leq s}\min(X(u), 0)\right). \end{eqnarray*}  Let $\tau_{0} = \inf\{t \geq 0:  X(t) = 0\}$.  If $\tau_{0} > t$, then
\begin{eqnarray*} R(t) - \inf_{0 \leq s \leq t}R(s) &=& X(t) - \inf_{0 \leq s \leq t}X(s) \end{eqnarray*} since $\min(X(s), 0) = 0$ for $0 \leq s \leq \tau_{0}$.  Next, if $\tau_{0} \leq t$, we also see that
\begin{eqnarray*} R(t) - \inf_{0  \leq s \leq t}R(s) = X(t) - \inf_{0 \leq s \leq t}X(s) - \inf_{\tau_{0} \leq s \leq t}\left(X(s) - \inf_{\tau_{0} \leq u \leq s}X(u)\right) = X(t) - \inf_{0 \leq s \leq t}X(s)\end{eqnarray*} since $\inf_{\tau_{0} \leq s \leq t}\left(X(s) - \inf_{\tau_{0} \leq u \leq s}X(u)\right) \geq 0$, and $X(\tau_{0}) - \inf_{\tau_{0} \leq u \leq \tau_{0}}X(u) = 0$.  Moreover, for each $t \geq 0$
\begin{eqnarray*} \inf_{0 \leq s \leq t}R(t) = \max(\inf_{0 \leq s \leq t}X(s), 0). \end{eqnarray*}  Thus, for an exponential random variable $e_q$ with parameter $q > 0$, independent of $X$, we have
\begin{eqnarray*} E_{x}[e^{i\omega(R(e_q) - \inf_{0 \leq s \leq e_q}R(s))}e^{i\omega\inf_{0 \leq s \leq e_q}R(s)}] &=& \int_{0}^{\infty}E_{x}[e^{i\omega(R(t) - \inf_{0 \leq s \leq t}R(s))}e^{i\omega\inf_{0 \leq s \leq t}R(s)}]qe^{-qt}dt \\ &=& \int_{0}^{\infty}E_{x}[e^{i\omega(X(t) - \inf_{0 \leq s \leq t}X(s))}e^{i\omega\max(0, \inf_{0 \leq s \leq t}X(s))}]qe^{-qt}dt \\ &=& E_{x}[e^{i\omega(X(e_q) - \inf_{0 \leq s \leq e_q}X(s))}e^{i\omega\max(0, \inf_{0 \leq s \leq e_q}X(s))}] \\ &=& E_{x}[e^{i\omega(X(e_q) - \inf_{0 \leq s \leq e_q}X(s))}]E_{x}[e^{i\omega\max(0, \inf_{0 \leq s \leq e_q}X(s))}] \end{eqnarray*} where the last step follows from the Wiener-Hopf factorization, i.e. Theorem \ref{classicalLevyThm}.

Theorem \ref{reflectionLevyThm} does not seem to be explicitly known, however direct computations of $E_{x}[e^{i \omega R(e_q)}]$ have appeared in various places:  see e.g. Theorem 9.1 of Abate and Whitt \cite{AbateWhittRBM2}, Theorem 2.1 of Abate and Whitt \cite{AbateWhittMM12}, Bingham \cite{Bingham}, Bekker et al. \cite{BekkerBoxmaResing}, and Chapter 9, Theorem 3.10 of Asmussen \cite{Asmussen}, where all of these references address the factorization in the case where $X$ is spectrally positive, i.e. $X$ has only positive jumps.  Theorem \ref{reflectionLevyThm} is also implicitly stated in Example 3 of Palmowski and Vlasiou \cite{PalmowskiVlasiou}, in terms of the steady-state distribution of a reflected L\'evy process that experiences catastrophes at times forming a homogeneous Poisson process.  Their result, like previous references, considers only the spectrally positive case, but their arguments can also be used to establish Theorem 3.4 as well.  Other results similar to Theorem \ref{reflectionLevyThm} can also be found in the recent work of Debicki et al. \cite{DebickiKosinskiMandjes}, and in Kella and Mandjes \cite{KellaMandjes}.

\section{Applications to birth-death processses, and diffusions}

We now apply our factorization identities, i.e. Theorems \ref{mainresult} and \ref{mainreflectedresult}, towards the study of birth-death processes, which form another interesting subclass of PRP systems.  It will also be possible to apply our identity towards the study of diffusion processes as well, as these are often weak limits of birth-death processes.

Readers should note that the transforms derived below can also be modified so that the domain is complex-valued, as we noted in the remark following Theorem \ref{mainresult} above.

\subsection{Birth-death processes}

Suppose that $Q := \{Q(t); t \geq 0\}$ represents a birth-death process on the integers, with birth rates $\{\lambda_{n}\}_{n \in \mathbb{Z}}$ and death rates $\{\mu_{n}\}_{n \in \mathbb{Z}}$.  Let $e_q$ represent an exponential random variable with rate $q > 0$, independent of $Q$.  Throughout we assume that $Q$ is ergodic, and we let $\pi$  represent its stationary distribution.  Our object of study is now the probability mass function of $Q(e_q)$.  We remind readers that $Q$ can easily be related to a PRP system:  units arrive according to a collection of independent Poisson processes $\{A_{0,j}\}_{j \in \mathbb{Z}}$ where $A_{0,j}$ has rate $\lambda_j$, each customer brings to the system a unit exponential amount of work, and the server processes work at a rate $\mu_n$ whenever the system is in state $n$, for $n \in \mathbb{Z}$.

By Corollary 4.1.1 of Abate and Whitt \cite{AbateWhittMM12}, we see that for each $n \in \mathbb{Z}$,
\begin{eqnarray*} P_{0}(Q(e_q) = n) = \frac{\pi_{n}E_{n}[e^{-q \tau_{0}}]}{\sum_{k \in \mathbb{Z}}\pi_{k}E_{k}[e^{-q \tau_{0}}]} \end{eqnarray*} where $P_{n}$ is meant to represent a conditional probability, given $Q(0) = n$.  This expression also holds in the absence of ergodicity, and also for complex $q$ when $P_{0}(Q(e_q) = n)$ is interpreted as a Laplace transform, multiplied by $q$.

However, suppose we would like to change the initial condition.  While the same method will tell us that
\begin{eqnarray*} P_{n_0}(Q(e_q) = n) = \frac{\pi_{n}E_{n}[e^{-q \tau_{n_0}}]}{\sum_{j \in \mathbb{Z}}\pi_{j}E_{j}[e^{-q \tau_{n_0}}]} \end{eqnarray*} for an arbitrary $n_0$, we must be careful:  how do we know that $E_{n}[e^{-q \tau_{n_0}}]$ is tractable?  This is a very legitimate question, as there are many instances where $E_{n}[e^{-q \tau_{n_0}}]$ will be tractable for some choices of $n_0$, but not for others.

Thus, the key to computing these probabilities is to choose the appropriate \emph{reference point}, i.e. the point found in the hitting-time Laplace-Stieltjes transforms given in the pmf of $Q(e_q)$.  This is where our factorization identities become useful:  they allow us to use whatever reference point we like, regardless of the initial value.

We illustrate our approach by computing the pmf of the number of customers in an $M/M/s$ queueing system at an independent exponential time $e_q$.  The reader will see that our expressions will be given in terms of an $M/M/1$ model and an $M/M/\infty$ model, which are much simpler.

\subsubsection{The $M/M/s$ queue}

Recall that the $M/M/s$ queue is a birth-death process on $\{0,1,2, \ldots\}$ with birth rates $\lambda_{n} = \lambda$, for $n \geq 0$, and death rates $\mu_{n} = \min\{n,s\}\mu$, for $n \geq 1$.  A classical reference on the time-dependent behavior of the $M/M/s$ queue is Saaty \cite{Saaty}, which makes use of the approach found in Bailey \cite{Bailey}.

Assume first that $Q(0) = s$.  In this case, for each $n \geq 0$,
\begin{eqnarray*} P_{s}(Q(e_q) = n) = \frac{\pi_{n}E_{n}[e^{-q \tau_{s}}]}{\sum_{j \geq 0}\pi_{j}E_{j}[e^{-q \tau_{s}}]}. \end{eqnarray*}  This is a nice expression:  notice that if $k < s$, $E_{k}[e^{-q \tau_{s}}]$ is the Laplace-Stieltjes transform of the amount of time it takes an $M/M/s$ queue to go from level $k$ to level $s$, but this is the same as the Laplace-Stieltjes transform of the amount of time it takes to go from $k$ to $s$ in an $M/M/\infty$ queue, with arrival rate $\lambda$ and service rate $\mu$.  Similarly, for $k > s$, $E_{k}[e^{-q \tau_{s}}]$ is just the LST of the amount of time it takes to go from level $k$ to level $s$ in an $M/M/1$ queue, with arrival rate $\lambda$ and service rate $s\mu$.  Hence, all of the terms in our expression for $P_{s}(Q(e_q) = k)$ can theoretically be derived from two simpler models, the $M/M/1$ queue and the $M/M/\infty$ queue.

For $k > s$, we already have a closed-form expression for $E_{k}[e^{-q \tau_{s}}]$:  letting $\psi(q) = E_{s+1}[e^{-q \tau_{s}}]$ be the busy period of an $M/M/1$ queue with arrival rate $\lambda$ and service rate $s\mu$, we see that
\begin{eqnarray*} E_{k}[e^{-q \tau_{s}}] = \psi(q)^{k-s}. \end{eqnarray*}
We now focus on the case where $k < s$.  Letting $\{Q_{M/M/\infty}(t); t \geq 0\}$ represent the queue-length process of an $M/M/\infty$ queue (including the customers in service), we use a classical argument found in Darling and Siegert \cite{DarlingSiegert} to find that
\begin{eqnarray*} P_{k}(Q_{M/M/\infty}(e_q) = s) &=& P_{k}(Q_{M/M/\infty}(e_q) = s, \tau_{s} \leq e_q) \\ &=& P_{s}(Q_{M/M/\infty}(e_q) = s)E_{k}[e^{-q \tau_{s}}] \end{eqnarray*} giving
\begin{eqnarray} \label{hittingtimeratio} E_{k}[e^{-q \tau_{s}}] = \frac{P_{k}(Q_{M/M/\infty}(e_q) = s)}{P_{s}(Q_{M/M/\infty}(e_q) = s)}. \end{eqnarray}

To compute $P_{k}(Q_{M/M/\infty}(e_q) = s)$, we need to use the following known lemma.  The $\mu = 1$ case was observed in Flajolet and Guillemin \cite{FlajoletGuillemin}, but we repeat it here for convenience.
\begin{lemma} \label{KummerLemma} For a positive real number $q$, \begin{eqnarray*} \int_{0}^{\infty}qe^{-(qt + \rho(1 - e^{-\mu t}))}dt &=& M\left(1, \frac{q}{\mu} + 1, -\rho \right) \end{eqnarray*} where $M$ is Kummer's function, i.e.
\begin{eqnarray*} M(a,b,z) = \sum_{n=0}^{\infty}\frac{(a)_n z^{n}}{(b)_n n!} \end{eqnarray*} with $(a)_0 = 1$, and for $n \geq 1$, $(a)_n = (a)(a + 1)\cdots(a + n - 1)$. \end{lemma}
\begin{proof} Applying partial integration gives
\begin{eqnarray*} \int_{0}^{\infty}e^{-\rho (1 - e^{-\mu t})}qe^{-q t}dt &=& 1 - \rho \mu \int_{0}^{\infty}e^{-(q + \mu)t}e^{-\rho(1 - e^{-\mu t})}dt. \end{eqnarray*}  After repeatedly applying partial integration and taking limits, we get the result. \end{proof}

\begin{lemma} For each $k \leq s$,
\begin{eqnarray*} P_{k}(Q_{M/M/\infty}(e_q) = s) &=& \sum_{j=0}^{k}\sum_{m=0}^{k + s - 2j}{k \choose j}{k + s - 2j \choose m}\frac{(\rho)^{s-j}(-1)^{m}}{(s-j)!}\frac{q}{q + (j+m)\mu}M\left(1, \frac{q}{\mu} + j + m + 1, -\rho\right). \end{eqnarray*} \end{lemma}
\begin{proof} This identity can be derived from the known fact that, at a fixed time $t \geq 0$, $Q(t)$ is the convolution of a binomial random variable with parameters $(k,e^{-\mu t})$ and a Poisson random variable with parameter $\rho(1 - e^{-\mu t})$.  The result then follows by integrating the pmf of $Q(t)$, and applying Lemma \ref{KummerLemma}.   
\end{proof}  By making use of this lemma in equation (\ref{hittingtimeratio}), we arrive at the following result.
\begin{lemma} \label{MMinftyhittingtimelemma} For each $k \leq s$, we see that
\begin{eqnarray*} E_{k}[e^{-q \tau_{s}}] = \frac{\sum_{j=0}^{k}\sum_{m=0}^{k + s - 2j}{k \choose j}{k + s - 2j \choose m}\frac{(\rho)^{s-j}(-1)^{m}}{(s-j)!}\frac{q}{q + (j+m)\mu}M\left(1, \frac{q}{\mu} + j + m + 1, -\rho\right)}{\sum_{j=0}^{s}\sum_{m=0}^{2(s-j)}{s \choose j}{2(s-j) \choose m}\frac{(\rho)^{s-j}(-1)^{m}}{(s-j)!}\frac{q}{q + (j+m)\mu}M\left(1, \frac{q}{\mu} + j + m + 1, -\rho\right)}. \end{eqnarray*} \end{lemma}

\begin{remark} As discussed in the remark following Theorem \ref{mainresult}, Lemmas 4.1, 4.2 and 4.3 can be modified so that $q$ is allowed to take on complex values. \end{remark}

Our next step is to use the Wiener-Hopf identity to compute probabilities of the form $P_{k}(Q(e_q) = n)$, for arbitrary $k,n \geq 0$.  Notice that we already have a nice expression for such a pmf, when $k = s$.

\noindent \textbf{Case 1}: $k > s$, $n \leq s$.  Notice that
\begin{eqnarray*} P_{k}(Q(e_q) = n) &=& P_{k}(Q(e_q) = n, \tau_{s} \leq e_q) \\ &=& P_{k}(Q(e_q) = n \mid \tau_{s} \leq e_q)E_{k}[e^{-q \tau_{s}}] \\ &=& P_{s}(Q(e_q) = n) E_{k}[e^{-q \tau_{s}}] \end{eqnarray*} showing, from our previous calculations, that this probability is tractable.  Readers should again note that a similar argument can be made for complex $q = x + iy$ satisfying $x > 0$.  Here
\begin{eqnarray*} \int_{0}^{\infty}P_{k}(Q(t) = n)qe^{-qt}dt &=& \frac{x + iy}{x}E_{k}[e^{-iye_{x}}\textbf{1}(Q(e_{x}) = n)] \\ &=& \frac{x + iy}{x}E_{k}[e^{-iye_{x}}\textbf{1}(Q(e_{x}) = n)\textbf{1}(\tau_{s} \leq e_{x})] \\ &=& \frac{x + iy}{x}E_{k}[e^{-iy(e_{x} - \tau_{s} + \tau_{s})}\textbf{1}(Q(e_{x} - \tau_{s} + \tau_{s}) = n)\textbf{1}(e_{x} \geq \tau_{s})] \\ &=& \frac{x + iy}{x}E_{s}[e^{-iye_{x}}\textbf{1}(Q(e_{x}) = n)]E_{k}[e^{-iy\tau_{s}}\textbf{1}(\tau_{s} \leq e_{x})] \\ &=& E_{k}[e^{-q\tau_{s}}]\int_{0}^{\infty}P_{s}(Q(t) = n)qe^{-qt}dt \end{eqnarray*} where the fourth equality holds by the strong Markov property.

\noindent \textbf{Case 2}:  $k > s$, $n > s$.  This case is much more interesting, since it is possible for our process to go from $k$ to $n$, without ever reaching level $s$ in $[0,e_q]$.  Proceeding in the same manner as in Case 1 yields
\begin{eqnarray*} P_{k}(Q(e_q) = n) &=& P_{k}(Q(e_q) = n, \tau_{s} \leq e_q) + P_{k}(Q(e_q) = n, \tau_{s} > e_q) \\ &=& P_{s}(Q(e_q) = n)E_{k}[e^{-q \tau_{s}}] + \sum_{l=s+1}^{\min\{n,k\}}P_{k}(Q(e_q) = n \mid \inf_{0 \leq u \leq e_q}Q(u) = l)P_{k}(\inf_{0 \leq u \leq e_q}Q(u) = l). \end{eqnarray*}  These terms are computable:  first note that
\begin{eqnarray*} P_{k}(\inf_{0 \leq u \leq e_q}Q(u) = l) &=& P_{k}(\tau_{l} \leq e_q) - P_{k}(\tau_{l-1} \leq e_q) \\ &=& E_{k}[e^{-q \tau_{l}}] - E_{k}[e^{-q \tau_{l-1}}] \\ &=& \psi(q)^{k-l} - \psi(q)^{k - l + 1} \end{eqnarray*} and from Theorem \ref{mainreflectedresult}, we find that conditional on $\inf_{0 \leq u \leq e_q}Q(u) = l$, $Q$ behaves as an $M/M/1$ queue on $[0,e_q]$ with arrival rate $\lambda$ and service rate $s\mu$.  Hence,
\begin{eqnarray*} P_{l}(Q(e_q) = n \mid \inf_{0 \leq u \leq e_q}Q(u) = l) = \left(1 - \frac{\lambda \psi(q)}{s\mu}\right)\left(\frac{\lambda \psi(q)}{s \mu}\right)^{n - l}. \end{eqnarray*}

\noindent \textbf{Case 3}: $0 \leq k < s$, $n \geq s$.  This case is analogous to Case 1:  here
\begin{eqnarray*} P_{k}(Q(e_q) = n) = P_{s}(Q(e_q) = n)E_{k}[e^{-q \tau_{s}}]. \end{eqnarray*}  Now we can use Lemma \ref{MMinftyhittingtimelemma} to express $E_{k}[e^{-q \tau_{s}}]$ in terms of Kummer functions.

\noindent \textbf{Case 4}: $0 \leq k < s$, $n < s$.  As expected, this case is analogous to Case 2, but the expression here is more complicated than the other cases.  Here
\begin{eqnarray*} P_{k}(Q(e_q) = n) &=& P_{k}(Q(e_q) = n, \tau_{s} \leq e_q) + P_{k}(Q(e_q) = n, \tau_{s} > e_q) \\ &=& P_{s}(Q(e_q) = n)E_{k}[e^{-q \tau_{s}}] + \sum_{l=\max\{k,n\}}^{s-1}P_{k}(Q(e_q) = n \mid \sup_{0 \leq u \leq e_q}Q(u) = l)P_{k}(\sup_{0 \leq u \leq e_q}Q(u) = l). \end{eqnarray*}  However, we again observe that
\begin{eqnarray*} P_{k}(\sup_{0 \leq u \leq e_q}Q(u) = l) = E_{k}[e^{-q \tau_{l}}] - E_{k}[e^{-q \tau_{l+1}}] \end{eqnarray*} and conditional on $\sup_{0 \leq u \leq e_q}Q(u) = l$, we use Theorem \ref{mainreflectedresult} to deduce that $Q$ behaves as an $M/M/l/l$ queue on $[0,e_q]$, starting at level $l$.  This yields
\begin{eqnarray*} P_{k}(Q(u) = n \mid \sup_{0 \leq u \leq e_q}Q(u) = l) &=& \frac{\frac{\rho^{n}}{n!}E_{n}[e^{-q \tau_{l}}]}{\sum_{j=0}^{l}\frac{\rho^{j}}{j!}E_{j}[e^{-q \tau_{l}}]} \end{eqnarray*} implying that this final case is tractable as well, in that it can be expressed in terms of Kummer functions.

There is an important lesson to be learned from our calculations of the pmf of $Q(e_q)$.  Given a proper choice of initial point and reference point, our probability mass function of $Q(e_q)$ can be expressed in terms of quantities related to three simpler models:  the $M/M/1$ queue, the $M/M/l/l$ queue, and the $M/M/\infty$ queue.  Had we chosen another reference point different from $s$, our hitting-time transforms would have been much more difficult to compute.

\subsubsection{The $M/M/s/K$ queue}

Our factorization identities can also be used to derive the pmf of the $M/M/s/K$ queue-length process at an independent exponential time $e_q$, where $s$ is the number of servers and $K$ the system capacity.  By choosing our reference point to be $s$, we mimic the procedure used in the $M/M/s$ case to express the desired pmf in terms of two simpler models: the $M/M/s/s$ queue (which is expressible in terms of $M/M/\infty$ hitting-time transforms), and the $M/M/1/(K-s)$ queue.

Note that the relevant hitting-time transforms for the $M/M/1/(K-s)$ queue can be derived from the $M/M/1$ queue, since we can use the pmf of an $M/M/1$ queue at an exponential time to derive the LST of the time it takes us to go from level $j_1$ to level $j_2$ in an $M/M/1$ queue, when $j_1 < j_2$.  Such a result can then be used to derive all of the corresponding hitting-time transforms for an $M/M/1/(K-s)$ queue.

\subsubsection{Time-dependent moments}

It is possible to make use of the factorization identities to derive the moments of $Q(e_q)$ as well.  To illustrate the main idea, we first suppose that $\{Q(t); t \geq 0\}$ represents an $M/M/1$ queue-length process, with arrival rate $\lambda$ and service rate $\mu$.  It has been shown in Abate and Whitt \cite{AbateWhittMM11} that, for each $t \geq 0$,
\begin{eqnarray*} E[Q(t) \mid Q(0) = 0] = \frac{\rho}{1 - \rho}P(R_{\tau} \leq t) \end{eqnarray*} where $\tau$ represents the busy period of an $M/M/1$ queue, and $R_{\tau}$ represents the residual busy period, i.e. for each $t > 0$,
\begin{eqnarray*} P(R_{\tau} > t) = \frac{1}{E[\tau]}\int_{t}^{\infty}P(\tau > x)dx. \end{eqnarray*}  Letting $e_q$ be an exponential r.v. with rate $q > 0$, independent of $Q$, gives
\begin{eqnarray*} E[Q(e_q) \mid Q(0) = 0] &=& \frac{\rho}{1 - \rho}E[e^{-q R_{\tau}}] \\ &=& \frac{\rho}{1 - \rho}\frac{1 - E[e^{-q \tau}]}{qE[\tau]} \\ &=& \frac{\lambda(1 - E[e^{-q \tau}])}{q} \end{eqnarray*} which implies that the first moment of $Q(e_q)$ is tractable, assuming we start in state 0.

Our factorization identities can now be used to compute the first moment of $Q(e_q)$, for any initial condition.  Suppose that $Q(0) = n_0 \geq 0$.  Then
\begin{eqnarray*} E[Q(e_q) \mid Q(0) = n_0] &=& E[Q(e_q) \mid \inf_{0 \leq s \leq e_q}Q(s) = 0, Q(0) = n_0]P(\inf_{0 \leq s \leq e_q}Q(s) = 0 \mid Q(0) = n_0) \\ &+& \sum_{k=0}^{n_0}E[Q(e_q) \mid \inf_{0 \leq s \leq e_q}Q(s) = k, Q(0) = n_0]P(\inf_{0 \leq s \leq e_q}Q(s) = k \mid Q(0) = n_0) \\ &=& E[Q(e_q) \mid Q(0) = 0]P(\inf_{0 \leq s \leq e_q}Q(s) = 0 \mid Q(0) = n_0) \\ &+& \sum_{k=0}^{n_0}(E[Q(e_q) \mid Q(0) = 0] + k)P(\inf_{0 \leq s \leq e_q}Q(s) = k \mid Q(0) = n_0) \\ &=& E[Q(e_q) \mid Q(0) = 0] + \sum_{k=0}^{n_0}k P(\inf_{0 \leq s \leq e_q}Q(s) = k \mid Q(0) = n_0) \\ &=& \frac{\lambda(1 - E[e^{-q \tau}])}{q} + \sum_{k=1}^{n_0}k \psi(q)^{n_0 - k}(1 - \psi(q)). \end{eqnarray*}  The key step in this derivation is the second equality:  if $\inf_{0 \leq s \leq e_q}Q(s) = k$, then Theorem \ref{mainreflectedresult} tells us that $Q(e_q)$ is equal in distribution to an $M/M/1$ queue on the states $\{k, k+1, k+2, \ldots\}$ with arrival rate $\lambda$ and service rate $\mu$.  This result agrees with the result given in \cite{AbateWhittMM12}, and also in \cite{FralixRiano}.  With a bit of patience, higher moments can also be computed through the use of this approach, but there are better ways to do this for the $M/M/1$ model:  see \cite{FralixRiano} for details.

An analogous procedure can be used to compute the moments of $Q(e_q)$, for more complicated processes.  Suppose now that $\{Q(t); t \geq 0\}$ represents the queue-length process of an $M/M/s$ queue, with arrival rate $\lambda$ and service rate $\mu$, and $s$ servers.  While the transient moments of the $M/M/s$ queue have been studied in Marcell\'an and P\'erez \cite{MarcellanPerez}, the point here is to show how to construct the moments from simpler birth-death processes.

The key to computing the moments of $Q(e_q)$ for an arbitrary initial condition is to first compute the moments, while assuming that $Q(0) = s$, since we will want to again use $s$ as a reference point when we apply Theorem \ref{mainreflectedresult}.  Again, since $Q$ is a reversible process, we can say that
\begin{eqnarray*} E[Q(e_q) \mid Q(0) = s] = \pi_{0}(q)\sum_{k = 0}^{s}k E_{k}[e^{-q \tau_{s}}]\frac{\rho^{k}}{k!} + \pi_{0}(q)\frac{\rho^{s}}{s!}\sum_{k = s+1}^{\infty}k E_{k}[e^{-q \tau_{s}}](\rho/s)^{k-s} \end{eqnarray*} with
\begin{eqnarray*} \pi_{0}(q) = \left[\sum_{k=0}^{s}E_{k}[e^{-q \tau_s}]\frac{(\rho)^{k}}{k!} + \sum_{k = s + 1}^{\infty}\frac{(\rho)^{s}}{s!}\left(\rho/s\right)^{k-s}E_{k}[e^{-q \tau_s}]\right]^{-1} \end{eqnarray*} being the normalizing constant.
There are a few observations here worth noting.  First, notice that
\begin{eqnarray*} \pi_{0}(q)\sum_{k=0}^{s}k E_{k}[e^{-q \tau_{s}}]\frac{\rho^{k}}{k!} &=& P_{s}(Q(e_q) \leq s)E_{s}[Q_{M/M/s/s}(e_q)] \end{eqnarray*} where $Q_{M/M/s/s}$ represents an $M/M/s/s$ loss model with arrival rate $\lambda$, service rate $\mu$, and $s$ servers, and this is a known expected value; see Abate and Whitt \cite{AbateWhittErlangLoss} for details.  Second, we see that
\begin{eqnarray*} \pi_{0}(q)\frac{\rho^{s}}{s!}\sum_{k = s+1}^{\infty}k E_{k}[e^{-q \tau_{s}}](\rho/s)^{k-s} &=& \pi_{0}(q)\frac{\rho^{s}}{s!}\sum_{k = s+1}^{\infty}(k-s) E_{k}[e^{-q \tau_{s}}](\rho/s)^{k-s} \\ &+& \pi_{0}(q)\frac{\rho^{s}}{s!}\sum_{k = s+1}^{\infty}s E_{k}[e^{-q \tau_{s}}](\rho/s)^{k-s} \\ &=& P_{s}(Q(e_q) \geq s)E_{0}[Q_{M/M/1}(e_q)] + sP_{s}(Q(e_q) \geq s)P_{0}(Q_{M/M/1}(e_q) \geq 1) \end{eqnarray*} where $Q_{M/M/1}$ represents an $M/M/1$ queue with arrival rate $\lambda$ and service rate $s\mu$.  Thus, we conclude that $E[Q(e_q) \mid Q(0) = s]$ is a quantity that can be computed.

To get $E[Q(e_q) \mid Q(0) = i]$ for an arbitrary $i \geq 0$, we now invoke Theorem \ref{mainreflectedresult}.  Suppose first that $i < s$.  Then
\begin{eqnarray*} E[Q(e_q) \mid Q(0) = i] &=& \sum_{j=i}^{s-1}E[Q(e_q) \mid \sup_{0 \leq s \leq e_q}Q(s) = j, Q(0) = i]P(\sup_{0 \leq s \leq e_q}Q(s) = j \mid Q(0) = i) \\
&+& E[Q(e_q) \mid Q(0) = s]P(\tau_{s} \leq e_q) \end{eqnarray*} and we observe from Theorem \ref{mainreflectedresult} that, conditional on $\sup_{0 \leq s \leq e_q}Q(s) = j$, $Q(e_q)$ behaves as an $M/M/j/j$ queue on $\{0,1,2,\ldots,j\}$, meaning
\begin{eqnarray*} E[Q(e_q) \mid \sup_{0 \leq s \leq e_q}Q(s) = j, Q(0) = i] &=& E_{j}[Q_{M/M/j/j}(e_q)]. \end{eqnarray*}  All of the other terms in the sum are, for similar reasons, also tractable.  A similar argument can be used to derive $E[Q(e_q) \mid Q(0) = i]$ for $i > s$; we omit the details.

We also point out that a similar argument can be used to derive moment expressions for the $M/M/s$ queue with exponential reneging, i.e. the $M/M/s-M$ queue, which is the model studied in Garnett et al. \cite{GarnettMandelbaumReiman}.  Such moments would be decomposed into components from an $M/M/s/s$ queue, and a $M/M/1-M$ queue, and the $M/M/1-M$ queue moments have recently been studied in \cite{FralixReneging}.

\subsection{Diffusion processes}

The factorization identities can also be used to establish similar expressions for diffusion processes.  We illustrate how the procedure works by applying it to a classical reflected diffusion:  regulated Brownian motion.

\subsubsection{Regulated Brownian motion}

Suppose that $\{B(t); t \geq 0\}$ represents a Brownian motion, with drift $\mu = -1$ and volatility $\sigma^2 = 1$.  We are interested in understanding the time-dependent behavior of $\{R(t); t \geq 0\}$, where
\begin{eqnarray*} R(t) = B(t) - \inf_{0 \leq u \leq t}\min(B(u),0) \end{eqnarray*} i.e. $R$ is the one-sided reflection of $B$.  Granted, since $B$ is a L\'evy process, we can already use the Wiener-Hopf factorization to derive the Laplace-Stieltjes transform of $R(e_q)$.  However, we will instead be interested in showing how our factorization identities can also be used to derive the probability density function of $R(e_q)$.

To derive this pdf, we will need to know a bit about the distribution of the hitting times associated with a Brownian motion.  Following the classical argument of applying the optional sampling theorem to the Wald martingale, we see that
\begin{eqnarray*} E_{x}[e^{-q \tau_{0}}] = e^{-(-1 + \sqrt{1 + 2q})x}. \end{eqnarray*}  Moreover, $R$ has a unique stationary distribution $\pi$, where $\pi(dx) = 2e^{-2x}dx$.

We will now compute the density of $R(e_q)$, given $R(0) = x_0$:  we denote this density at the point $x$ as $f_{R(e_q)}(x; x_0)$.  Again, we will need to break the calculation up into cases.  Considering first the case where $x > x_0$, we may use Theorem \ref{mainreflectedresult}, along with a weak-convergence argument to show that
\begin{eqnarray*} P_{x_0}(R(e_q) > x) &=& E_{x_0}[e^{-q \tau_{0}}]\frac{\int_{x}^{\infty}E_{y}[e^{-q \tau_{0}}]\pi(dy)}{\int_{0}^{\infty}E_{y}[e^{-q \tau_{0}}]\pi(dy)} \\ &+& \int_{0}^{x_0} \frac{\int_{x}^{\infty}E_{y}[e^{-q \tau_{0}}]\pi(dy)}{\int_{z}^{\infty}E_{y}[e^{-q \tau_{0}}]\pi(dy)}dP(\inf_{0 \leq u \leq e_q}R(u) \leq z). \end{eqnarray*}  Careful readers will note that this identity is valid for a large class of reflected diffusion processes (namely, those processes that are expressible as a scaling-limit of a sequence of birth-death processes), not just for regulated Brownian motion.  Success in using this identity for a given diffusion depends on both the tractability of the hitting-time transforms, and the integrals containing them.

For $x \geq 0$, we can use our expressions for both the hitting-time LST and the stationary distribution to show that
\begin{eqnarray*} \int_{x}^{\infty}E_{y}[e^{-q \tau_{0}}]\pi(dy) &=& \int_{x}^{\infty}e^{-(-1 + \sqrt{1 + 2q})y}2e^{-2y}dy \\ &=& \frac{2}{1 + \sqrt{1 + 2q}}e^{-(1 + \sqrt{1 + 2q})x}. \end{eqnarray*}

Also, for $0 < z < x_0$,
\begin{eqnarray*} P_{x_0}(\inf_{0 \leq u \leq e_q}R(u) \leq z) &=& P_{x_0}(\tau_{z} \leq e_q) \\ &=& E_{x_0}[e^{-q \tau_{z}}] \\ &=& E_{x_0 - z}[e^{-q \tau_{0}}] \\ &=& e^{-(-1 + \sqrt{1 + 2q})(x_0 - z)} \end{eqnarray*} so for positive $z$, we find that the density of $\inf_{0 \leq u \leq e_q}R(u)$ is just
\begin{eqnarray*} dP(\inf_{0 \leq u \leq e_q}R(u) \leq z) = (-1 + \sqrt{1 + 2q})e^{-(-1 + \sqrt{1 + 2q})x_0}e^{(-1 + \sqrt{1 + 2q})z}dz. \end{eqnarray*}  Plugging everything in, we can now say that
\begin{eqnarray*} P_{x_0}(R(e_q) > x) &=& e^{-(-1 + \sqrt{1 + 2q})x_0}e^{-(1 + \sqrt{1 + 2q})x} \\ &+& \int_{0}^{x_0}e^{-(1 + \sqrt{1 + 2q})x}e^{(1 + \sqrt{1 + 2q})z}(-1 + \sqrt{1 + 2q})e^{-(-1 + \sqrt{1 + 2q})x_0}e^{(-1 + \sqrt{1 + 2q})z}dz \\ &=& e^{-(-1 + \sqrt{1 + 2q})x_0}e^{-(1 + \sqrt{1 + 2q})x}\left[1 + \frac{(-1 + \sqrt{1 + 2q})}{2\sqrt{1 + 2q}}\left[e^{2\sqrt{1 + 2q}x_0} - 1\right]\right] \end{eqnarray*} and so after taking derivatives and multiplying by $(-1)$, we find that the transient density of $R(e_q)$, for $x > x_0$, is just
\begin{eqnarray*} f_{R(e_q)}(x; x_0) &=& (1 + \sqrt{1 + 2q})e^{-(-1 + \sqrt{1 + 2q})x_0}e^{-(1 + \sqrt{1 + 2q})x} \\ &+& \frac{q}{\sqrt{1 + 2q}}e^{-(-1 + \sqrt{1 + 2q})x_0}e^{-(1 + \sqrt{1 + 2q})x}\left[e^{2\sqrt{1 + 2q}x_0} - 1\right]. \end{eqnarray*}

We will now focus on computing $f_{R(e_q)}(x; x_0)$, for $x < x_0$.  After applying our weak-convergence results, we see that
\begin{eqnarray*} P_{x_0}(R(e_q) > x) &=& 1 - E_{x_0 - x}[e^{-q \tau_{0}}] + E_{x_0}[e^{-q \tau_{0}}] \frac{\int_{x}^{\infty}E_{y}[e^{-q \tau_{0}}]\pi(dy)}{\int_{0}^{\infty}E_{y}[e^{-q \tau_{0}}]\pi(dy)} \\ &+& \int_{0}^{x}\frac{\int_{x}^{\infty}E_{y}[e^{-q \tau_{0}}]\pi(dy)}{\int_{z}^{\infty}E_{y}[e^{-q \tau_{0}}]\pi(dy)}dP_{x_0}(\inf_{0 \leq u \leq e_q}R(u) \leq z). \end{eqnarray*}  Evaluating this quantity, then taking derivatives shows that the transient density of $R(e_q)$ is just
\begin{eqnarray*} f_{R(e_q)}(x;x_0) &=& (-1 + \sqrt{1 + 2q})e^{-(-1 + \sqrt{1 + 2q})x_0}e^{-(1 - \sqrt{1 + 2q})x} \\ &+& (1 + \sqrt{1 + 2q})e^{-(-1 + \sqrt{1 + 2q})x_0}e^{-(1 + \sqrt{1 + 2q})x} \\ &+& (1 - \sqrt{1 + 2q})(-1 + \sqrt{1 + 2q})e^{-(-1 + \sqrt{1 + 2q})x_0}e^{-(1 - \sqrt{1 + 2q})x} \\ &-& \frac{q}{\sqrt{1 + 2q}}e^{-(-1 + \sqrt{1 + 2q})x_0}e^{-(1 + \sqrt{1 + 2q})x} \\ &=& (-1 + \sqrt{1 + 2q})(2 - \sqrt{1 + 2q})e^{-(-1 + \sqrt{1 + 2q})x_0}e^{-(1 - \sqrt{1 + 2q})x} \\ &+& \frac{\sqrt{1 + 2q} + 1 + q}{\sqrt{1 + 2q}}e^{-(-1 + \sqrt{1 + 2q})x_0}e^{-(1 + \sqrt{1 + 2q})x}. \end{eqnarray*}


\appendix

\section{Palm measures}

Throughout this paper, we assume that all of our random elements reside on a probability space $(\Omega, \mathcal{F}, P)$, where $\Omega$ represents a complete, separable metric space, $\mathcal{F}$ the Borel $\sigma$-field generated by the open sets of the metric, and $P$ a probability measure on $\mathcal{F}$.  These additional restrictions will be needed in order to properly define a collection of Palm measures, which are used to derive our main result.  The reader should not be alarmed by such restrictions, as the space $D[0, \infty)$ endowed with the proper choice of Skorohod metric is a complete, separable metric space, and many queueing processes (and stochastic processes in general) can reside on such a space.  Moreover, $\mathbb{R}_{+}$ is used to represent the nonnegative real line, and $\mathcal{B}$ the Borel $\sigma$-field generated by the open sets of $\mathbb{R_{+}}$.

Let $N := \{N(t); t \geq 0\}$ represent a point process on the nonnegative real line, with mean measure $\mu$, where $\mu(A) = E[N(A)] < \infty$ for all bounded $A \in \mathcal{B}$.  Under such assumptions, it is known that $N$ induces a $\mu$-a.e. unique  probability kernel $\mathcal{P}:  \mathbb{R}_{+} \times \mathcal{F} \rightarrow [0,1]$, where for each fixed $E \in \mathcal{F}$, $\mathcal{P}_{s}(E)$ is a Borel measurable function in $s$, and for each fixed $s \in \mathbb{R}_{+}$, $\mathcal{P}_{s}$ is a probability measure on $\mathcal{F}$.  The probability distributions of this kernel are referred to as the Palm measures of $N$, and these are defined to be the measures that satisfy the following condition:  for each $B \in \mathcal{B}$, and each $A \in \mathcal{F}$,
\begin{eqnarray} \label{Palmdefinition} E[N(B)\textbf{1}_{A}] = \int_{B}\mathcal{P}_{s}(A)\mu(ds). \end{eqnarray}
An important consequence of equation (\ref{Palmdefinition}) is the Campbell-Mecke formula; see for instance Kallenberg \cite{KallenbergRM}.  The proof of this formula follows from applying a monotone class argument to (\ref{Palmdefinition}).
\begin{theorem} (Campbell-Mecke formula) For any measurable stochastic process $\{X(t); t \geq 0\}$, we find that
\begin{eqnarray*} E\left[\int_{0}^{\infty}X(s)N(ds)\right] = \int_{0}^{\infty}\mathcal{E}_{s}[X(s)]\mu(ds) \end{eqnarray*} where $\mathcal{E}_{s}$ represents expectation, under the probability measure $\mathcal{P}_{s}$. \end{theorem}  Throughout, we say that a stochastic process is measurable if it is measurable with respect to the $\sigma$-field $\mathcal{A}$, which is generated by sets of the form $A \times C$, where $A \in \mathcal{B}$, and $C \in \mathcal{F}$, i.e. if for each $B \in \mathcal{B}$, $\{(t, \omega); X(t,\omega) \in B\} \in \mathcal{A}$.

The Campbell-Mecke formula is a very important, fundamental result in the theory of Palm measures, and is typically the main tool used when applying Palm measures to a given problem.  Readers wishing to consult a rigorous treatment of such measures are referred to Chapters 10-12 of \cite{KallenbergRM}:  other classical references on point process theory include the series of textbooks by Daley and Vere-Jones \cite{DaleyVereJones1, DaleyVereJones2}.

A collection of sub-$\sigma$-fields $\{\mathcal{F}_{s}; s \geq 0\}$ of $\mathcal{F}$ is said to be a filtration, if for each $s < t$, $\mathcal{F}_{s} \subset \mathcal{F}_{t}$.  We say that a stochastic process $\{X(t); t \geq 0\}$ is adapted to the filtration if, for each $t \geq 0$, $X(t)$ is measurable with respect to $\mathcal{F}_t$.  Associated with a filtration is a collection of $\sigma$-fields $\{\mathcal{F}_{s-}; s > 0\}$, where $\mathcal{F}_{s-}$ is the smallest $\sigma$-field containing all $\sigma$-fields $\mathcal{F}_{r}$, for $r < s$.  These are standard concepts within stochastic calculus, and can be found in virtually any textbook on the subject.  Some examples of textbooks that focus on point processes, and include such concepts, are Br\'emaud \cite{BremaudPP} and Baccelli and Br\'emaud \cite{BaccelliBremaud}.

We are now ready to quote a result that is used to derive the main result of this paper.  Suppose $N := \{N(t); t \geq 0\}$ represents a point process on $[0, \infty)$, and suppose $\{\mathcal{F}_{t}; t \geq 0\}$ represents a filtration, to which $N$ is adapted.  Within this framework, we say that $N$ is an $\mathcal{F}_{t}$-Poisson process, if (i) $N$ is adapted to the filtration, and (ii) the distribution of $N(a,b]$, conditional on $\mathcal{F}_{a}$, is Poisson with rate
\begin{eqnarray*} \mu(a,b] = \int_{(a,b]}\lambda(s)ds \end{eqnarray*} for some deterministic function $\lambda: [0, \infty) \rightarrow [0, \infty)$ (i.e. $N(a,b]$ is independent of $\mathcal{F}_{a}$).  Under these conditions, we can apply the following result, which is a corollary of a time-dependent analogue of Papangelou's lemma for point processes; see \cite{FralixRianoSerfozo} for details.
\begin{proposition} \label{ASTAprop} If $N$ is an $\mathcal{F}_{t}$-Poisson process, then $\mathcal{P}_t = P$ on $\mathcal{F}_{t-}$, for almost all $t$ (w.r.t. Lebesgue measure). \end{proposition}

\noindent \textbf{Acknowledgements} The authors would like to thank an anonymous referee for providing valuable comments on our paper, and for bringing reference \cite{Millar} to our attention.

\end{document}